\documentclass[12pt]{article} 

\usepackage{amsmath}
\usepackage{amsthm}
\usepackage{amsfonts}
\usepackage{mathrsfs}
\usepackage{stmaryrd}
\usepackage{setspace}
\usepackage{fullpage}
\usepackage{amssymb}
\usepackage{breqn}
\usepackage{enumitem}
\usepackage{bbold} 
\usepackage{authblk}
\usepackage{comment}
\usepackage{hyperref}
\usepackage{pgf,tikz}
\usepackage{graphicx}
\usetikzlibrary{decorations.pathreplacing,arrows}
\tikzstyle{vertex}=[circle,draw=black,fill=black,inner sep=0,minimum size=3pt,text=white,font=\footnotesize]

\newtheorem{thm}{Theorem}[section]

\newtheorem{lemma}[thm]{Lemma}

\newtheorem{proposition}[thm]{Proposition}

\newtheorem{corollary}[thm]{Corollary}

\newtheorem*{lemma*}{Lemma}
\newtheorem*{proposition*}{Proposition}
\newtheorem*{theorem*}{Theorem}
\newtheorem*{corollary*}{Corollary}

\newcommand\ex{\ensuremath{\mathrm{ex}}}

\newcommand\cN{{\mathcal N}}

\newcommand{\ignore}[1]{}

\title{On extremal values of some degree-based topological indices with a forbidden or a prescribed subgraph}
\author{Dániel Gerbner\\ \small Alfr\'ed R\'enyi Institute of Mathematics, HUN-REN\\
\small \texttt{gerbner.daniel@renyi.hu}}
\date{}

\begin{document}

\maketitle

\begin{abstract} 
Xu in 2011 determined the largest value of the second Zagreb index in an $n$-vertex graph $G$ with clique number $k$, and also the smallest value with the additional assumption that $G$ is connected. We extend these results to other degree-based topological indices. The key property of the clique number in the first result is that $G$ is $K_{k+1}$-free, while the key property in the second result is that $G$ contains a $K_{k+1}$. We also extend our investigations to other forbidden/prescribed subgraphs. Our main tool is showing that several degree-based topological indices are equal to the weighted sum of the number of some subgraphs of $G$.
\end{abstract}

\section{Introduction}


Various graph invariants based on the degrees of the vertices have been studied, especially in chemical graph theory. Extremal values of such invariants in certain classes of graphs have attracted lot of attention, see for example the surveys \cite{agmm, bdfg,gmm}. 

In this paper we deal with the following indices. The \textit{first Zagreb index} \cite{gutr} of a graph $G$ is $M_1(G)=\sum_{v\in V(G)} d(u)^2$. The \textit{forgotten topological index} \cite{gutr,fg} of a graph $G$ is $F(G)=\sum_{v\in V(G)} d(u)^3$. 

The \textit{second Zagreb index} \cite{grtw} of a graph $G$ is $M_2(G):=\sum_{uv\in E(G)}d(u)d(v)$. The \textit{reduced second Zagreb index} \cite{fge} is $RM_2(G)=\sum_{uv\in E(G)}(d(u)-1)(d(v)-1)$. The \textit{first hyper Zagreb index} \cite{srs} of $G$ is $HM_1(G)=\sum_{uv\in E(G)}(d(u)+d(v))^2$.  The \textit{second hyper Zagreb index} \cite{gfsj} of $G$ is $HM_2(G)=\sum_{uv\in E(G)}(d(u)d(v))^2$. The \textit{hyper $F$-index} \cite{gg} of a graph $G$ is $HF(G)=\sum_{uv\in E(G)}(d(u)^2+d(v)^2)^2$. The \textit{first reformulated Zagreb index} \cite{mnt} of a graph $G$ is $EM_1(G)=\sum_{uv\in E(G)}(d(u)+d(v)-2)^2$. The \textit{first K Banhatti index} \cite{kulli} of a graph $G$ is $B_1(G)=\sum_{uv\in E(G)}(3d(u)+3d(v)-4)$. The \textit{second K Banhatti index} \cite{kulli} of a graph $G$ is $B_2(G)=\sum_{uv\in E(G)}(d(u)+d(v)-2)(d(u)+d(v))$.

Let $r$ and $s$ be non-negative integers. The \textit{general sum-connectivity index} \cite{zt} of a graph $G$ is $\chi_r(G)=\sum_{uv\in E(G)}(d(u)+d(v))^r$. The \textit{general zeroth order Randi\'c index} or \textit{first general Zagreb
index} or \textit{variable first Zagreb index} of $G$ is $R^{(0)}_r(G)=\sum_{v\in V(G)} d(u)^r$, see \cite{cy}. The \textit{general Randi\'c index} \cite{be} of a graph $G$ is $R_r(G)\sum_{uv\in E(G)}(d(u)d(v))^r$. The \textit{generalized Zagreb index} \cite{ai} of a graph $G$ is $M_{r,s}(G)=\sum_{uv\in E(G)}(d(u)^rd(v)^s+d(u)^sd(v)^r)$. 
Note that in these indices usually $r$ and $s$ may take any real values; we consider only integer values in most of the paper. 

Let $A_0$ denote the set of these indices.
They are not arbitrarily chosen. They share the property that they can be written as the weighted sum of some subgraphs of $G$. 

Let $B_t(p,q)$ denote the following graph, which we call \textit{generalized book graph}. We take an edge $uv$, we add $t$ vertices joined to both $u$ and $v$, $p$ vertices joined to $u$ and $q$ vertices joined to $v$. We call $u$ and $v$ the \textit{rootlet} vertices, the $t$ vertices adjacent to $u$ and $v$ are the \textit{page vertices} and the $p+q$ vertices adjacent only to either $u$ or $v$ are the \textit{leaf vertices}. Note that $B_t(0,0)$ is the so-called book graph and $B_0(p,q)$ is the so-called double star $S_{p,q}$. 

Given graphs $H$ and $G$, we denote the number of (not necessarily induced) copies of $H$ in $G$ by $\cN(H,G)$.
We consider graph parameters of the type $\sum_{t,p,q\le k} w_{t,p,q}\cN(B_t(p,q),G)$. Let $A$ denote the set of these parameters with each $w_{t,p,q}\ge 0$.

\begin{proposition}\label{main}
    $A_0\subset A$. In other words, each of the indices listed above is equal to the weighted sum of the number of some generalized book graphs, with non-negative coefficients.
\end{proposition}

We note that this was observed for the second Zagreb index in \cite{gerbn} and for the general zeroth order Randic index in \cite{ger4}.

There are (at least) 4 ways to use the above proposition. We can use existing results on counting generalized books. Unfortunately, we are not aware of such results in this generality, but we can use results on counting double stars. We can use existing results on counting general graphs. We can prove new results on counting generalized books, in some cases it is easier to deal with them then with the degrees. Or we can adapt the methods used for counting graphs to our settings. We will show examples for each of these four ways. 

One of the most studied extremal problem concerning topological indices is the maximum, given the number of vertices and edges. For $m=\binom{a}{2}+b$ with $0\le b<a$ and $e\le \binom{n}{2}$, let the \textit{quasi-clique} $K_n^m$ be the graph consisting of a $K_a$, another vertex joined to $b$ vertices of the $K_a$, and $n-a-1$ further vertices, with no incident edges. Bollob\'as, Erd\H os and Sarkar \cite{bes} showed that among $n$-vertex graphs with $m$ edges, the quasi-clique has the largest $R_1$-value. We do not know such characterization for other $\alpha$ in the case of the general Randic index. We can show that if $m$ is sufficiently large compared to $n$, then the quasi-clique has asymptotically the largest value for any index in $A$.

\begin{proposition}
For each index $Q\in A$, there is a constant $c$ such that if $m\ge (1-c)\binom{n}{2}$, then for each $n$-vertex $G$ with $m$ edges we have that $Q(G)\le (1+o(1))Q(K_n^m)$.
\end{proposition}

 Here $o(1)$ means a sequence that goes to 0 as $n$ (thus $m$) goes to infinity. This proposition simply follows from a theorem of Gerbner, Nagy, Patk\'os and Vizer that states the following. For every graph $H$ there is a constant $c_H$ such that if $m\ge (1-c_H)\binom{n}{2}$, then the quasi-clique contains asymptotically the largest number of copies of $H$. Applying this to each $B_t(p,q)$ with $t,p,q\le k$ and using Proposition \ref{main}, we obtain 
 the above proposition.

 \smallskip

Xu \cite{xu} showed that in an $n$-vertex graph $G$ with clique number $k$, we have that $M_2(G)\le M_2(T(n,k))$, where $T(n,k)$ is the \textit{Tur\'an graph}, i.e., the complete $k$-partite $n$-vertex graph with each part of order $\lfloor n/k\rfloor$ or $\lceil n/k\rceil$. We give a new proof for this result. Observe that the only terms are of the form $d(u)d(v)$, thus $w_{t,p,q}\neq 0$ if and only if $2t+p+q\le 1$, i.e., we count a weighted sum of the graphs $P_2,P_3,P_4,K_3$. Here $P_\ell$ is the path on $\ell$ vertices. We know that among $n$-vertex $K_{k+1}$-free graphs, the Tur\'an graph contains the most copies of each of these subgraphs. This was shown for $P_2$ by Tur\'an \cite{T}, for $K_3$ (and other cliques) by Zykov \cite{zykov}, for $P_3$ by Gy\H ori, Pach and Simonovits \cite{gypl}. For $P_4$, this was shown in \cite{gypl} for $k=2$, in \cite{ger} for $k=3$ and in \cite{qxg} for any $k$. Therefore, Xu's result is implied.

The largest number of copies of $H$ in $F$-free graphs is denoted by $\ex(n,H,F)$ and is called the \textit{generalized Tur\'an number}. It has attracted a lot of attention since its sytematic study was initiated in \cite{AS}, but unfortunately we are not aware of results counting $B_t(p,q)$ in this setting either, except in the case $B_t(p,q)$ is a double star or a triangle.
For example, it is also known that if $k\ge 4$ and $F$ is a $(k+1)$-chromatic graph with a color-critical edge, i.e., an edge whose deletion decreases the chromatic number, then the number of copies of each of $P_2,P_3,P_4,K_3$ in $n$-vertex $F$-free graphs is maximized by $T(n,k)$ \cite{S,gerpal,hhl,zykov}, if $n$ is sufficiently large. This shows that $M_2$ is also maximized by $T(n,k)$. Let us remark here that several generalized Tur\'an results give the exact value of $\ex(n,H,F)$ only for sufficiently large $n$, often with an unspecified threshold. 

Given an index $Q$, let us denote by $\ex_Q(n,F)$ the largest value of $Q$ among $F$-free $n$-vertex graphs.
Let us summarize some of the immediate consequences of existing generalized Tur\'an results. Let $A_0'$ denote the set of indices in $A_0$ where each polynomial has degree at most 2 in each variable. This means each index in $A_0$ are in $A_0'$ except for the forgotten index, the hyper $F$-index and the indices with $r$ (or $r$ and $s$) in the definition. Those indices are in $A_0'$ if $r\le 2$ and $s\le 2$. The \textit{F\"uredi graph} \cite{fur} $F(n,t)$ is a $K_{2,t}$-free $n$-vertex graph with $(1-o(1))\ex(n,F)$ edges. For each of our indices $Q$, one can easily calculate $Q(F(n,t))$ asymptotically using that each degree in $F(n,t)$ is $(1+o(1))\sqrt{n}$.

\begin{thm}\label{gentu}

    \textbf{(i)} For any index $Q\in A$, if $k$ is large enough, then $\ex_Q(n,K_k)=Q(T(n,k-1))$.

    \textbf{(ii)} $\ex_{M_2}(n,P_k)=(1+o(1))Q(K_{\lfloor (k-1)/2\rfloor,n-\lfloor (k-1)/2\rfloor})=(1+o(1))(\lfloor (k-1)/2\rfloor)^2n^2$.

        \textbf{(iii)} $\ex_{M_2}(n,C_{2k})=(1+o(1))Q(K_{k-1,n-k+1})=(1+o(1))(k-1)^2n^2$.
    
\textbf{(iv)} $\ex_{M_2}(n,K_{2,t})=(1+o(1))Q(F(n,t))=(1+o(1))n^{5/2}/2$. 

\textbf{(v)} For any index $Q\in A_0'$, we have $\ex_{Q}(n,K_{2,t})=(1+o(1))Q(F(n,t))$.

    
\end{thm}

Combining results from \cite{ger2} and \cite{dahi}, we also have that \textbf{(i)} holds for any $k$-chromatic graph with a color-critical edge in place of $K_k$, if $n$ is large enough. Moreover, combining these with results from \cite{ger3}, we have exact results for some other $k$-chromatic graphs as well. We do not state these results precisely for brevity.

We continue with a theorem where we need to prove new generalized Tur\'an results, which are interesting on its own. Let $A_0''$ denote the set of indices in $A_0'$ without any term with $d(u)^2d(v)^2$. The first and second Zagreb indices, the reduced second Zagreb index, the first hyper Zagreb index, the first reformulated Zagreb index, the first and second K Banhatti indices, and the generalized Zagreb index with $r+s\le 3$ belong here.

\begin{thm}\label{gentur}
\textbf{(i)} We have $\ex(n,B_1(1,1),K_4)=\cN(B_1(1,1),(T(n,3))$.

\textbf{(ii)} We have $\ex(n,S_{1,2},K_4)=\cN(S_{1,2},(T(n,3))$.

    \textbf{(iii)} For any index $Q\in A_0'$, we have $\ex_Q(n,K_4)=Q(T(n,3))$.
\end{thm}


\smallskip

In the triangle-free case, we can extend the ideas used for counting double stars to many other indices.
A function $f(x,y)$ is \textit{monotone} if $x\le x'$ and $y\le y'$ implies $f(x,y)\le f(x',y')$. We say that a topological index is \textit{monotone increasing} if it can be written as $\sum_{uv\in E(G)}f(d(u),d(v))$ for a monotone function $f(x,y)$. Let $A_1$ denote the set of monotone increasing indices. Clearly $A\subset A_1$. Observe that the above definition is symmetric in the sense that for every edge $uv$ we have $f(d(u),d(v))$ and $f(d(v),d(u))$ in the sum.
\begin{thm}\label{trif}
    For any monotone increasing topological index $Q$, $\ex_Q(n,K_3)=Q(K_{m,n-m})$ for some $m$.
\end{thm}

We will use a stability version of the above theorem. 

\begin{lemma}\label{trifstab}
     For any index $Q$ in $A_1$, and any $\varepsilon>0$, there exists $\delta>0$ such that if $Q(G)\ge (1-\delta)\ex_Q(n,K_3)$ for an $n$-vertex triangle-free graph $G$, then $G$ can be turned to a complete bipartite graph by adding and removing at most $\epsilon n^2$ edges.
\end{lemma}

Now our goal is to extend this result to other 3-chromatic graphs. However, we have to place more restrictions on $f$. 
Let $A_1'$ denote the set of indices in $A_1$ with the additional property that $f$ is a polynomial. Then we have $A_0\subset A\subset A_1'\subset A_1$.

\begin{lemma}\label{ezis}
         For any index $Q$ in $A_1'$, any graph $F$ with chromatic number 3 and any $\varepsilon>0$, there exists $\delta>0$ such that if $Q(G)\ge (1-\delta)\ex_Q(n,F)$ for an $n$-vertex $F$-free graph $G$, then $G$ can be turned to a complete bipartite graph by adding and removing at most $\epsilon n^2$ edges.
\end{lemma}

\begin{corollary}\label{corro}
For any index $Q$ in $A_1'$ and any graph $F$ with chromatic number 3 and a color-critical edge, we have that $\ex_Q(n,F)=Q(K_{m,n-m})$ for some $m$, provided $n$ is sufficiently large.
\end{corollary}

Note that the above statement does not hold for every index in $A_1$. If for example $f$ has a large jump at $x+y=n+1$, it is possible $f(2,n-1)>n^2f(i,n-i)$ for every $i$ and then the $C_5$-free graph $B_1(n-3,0)$ (a triangle with leaves attached to one of the vertices) has larger $Q$-value than any complete bipartite graph $K_{i,n-i}$. One can easily see that Lemma \ref{ezis} also may fail in this case.

\smallskip
We have mentioned that Xu \cite{xu} determined the $n$-vertex graph $G$ with clique number $k$ that has the largest $M_2(G)$. He also determined the connected $n$-vertex graph $G'$ with clique number $k$ with the smallest $M_2(G')$ (obviously for disconnected graph we would not add any more edge to $K_k$). It is the \textit{kite graph} $Ki(n,k)$, which consists of a copy of $K_k$ and a path $P_{n-k}$ with an endpoint connected to a vertex of the clique.

We extend this result to every index in $A$ in the case $k>2$. Furthermore, we study the case the presence of another graph $H$ is known. 

\begin{thm}\label{klik}
    For any integers $t,p,q,k,n$, the connected $n$-vertex graph with clique number $k$ that contains the least number of copies of $B_t(p,q)$ is the kite graph $Ki(n,k)$, except for $k=2, t=0, p=q=1$, where the star $S_n$ contains the least number of copies of $S_{a,b}=P_4$.

\end{thm}

\begin{corollary}
    For every index $Q\in A$, if $G$ is an $n$-vertex connected graph with clique number $k>2$, then $Q(G)\ge Q(Ki(n,k))$.
\end{corollary}

The case $k=2$ is equivalent to finding the tree with the smallest $Q$-value. This has been studied for several indices, see e.g.
\cite{lz}.

We say that a graph $H$ is \textit{vertex-transitive} if for any $u,v\in V(G)$ there is an automorphism of $G$ that maps $u$ to $v$. Let $Ki(n,H)$ denote the graph we obtain from $H$ by adding a $P_{n-|V(H)|}$ and connecting an endpoint to a vertex of $H$. 

\begin{thm}\label{tranzi}
    Let $H$ be a connected vertex-transitive graph. For any integers $t,p,q,k,n$, the connected $n$-vertex graph  that contains $H$ and has the least number of copies of $B_t(p,q)$ is the kite graph $Ki(n,H)$, except for $H=K_2, t=0, p=q=1$, where the star $S_n$ contains the least number of copies of $S_{a,b}=P_4$.
\end{thm}

If $H$ is not vertex-transitive, then the same definition for generalized kite graphs results on non-isomorphic graphs, depending on what vertex of $H$ is chosen. Let $\mathcal{K}i(n,H)$ denote the family of such graphs. It is likely that a member of $\mathcal{K}i(n,H)$ contains the smallest number of copies of $B_t(p,q)$ among connected $n$-vertex graphs containing $H$, except for some small graphs $H$ and $B_t(p,q)=P_4$. However, this would not immediately solve the problem for indices in $A$ (or in $A_0$), since the actual extremal graph can be different for different values of $t,p,q$. Still, we believe that for indices in $A$, a member of $\mathcal{K}i(n,H)$ has the smallest $Q$-value among connected $n$-vertex graphs containing $H$, except for some small graphs $H$.

\section{Proofs}

We start with the proof of Proposition \ref{main}, which states that $A_0\subset A$.

\begin{proof}[Proof of Proposition \ref{main}]
    Observe that for each index in $A_0$, each term is a polynomial with variables $d(u)$ and $d(v)$. After expanding, each term is of the form $d(u)^ad(v)^b$ for some $a$ and $b$, where $uv\in E(G)$. We can consider this as taking neighbors of $u$ with repetition $a$ times and neighbors of $v$ with repetition $b$ times. This way we obtain a set $U$ of neighbors of $u$ and a set $V$ of neighbors of $v$. Let $t=|U\cap V|$, $p=U\setminus (V\cup \{v\})$ and $q=V\setminus (U\cup\{u\})$. Clearly the vertices in $U\cup\{u\}\cup V\cup\{v\}$ form a copy of $B_t(p,q)$.

    A copy of $B_t(p,q)$ with rootlet vertices $u$ and $v$ is counted multiple times: for a given term $d(u)^ad(v)^b$ with $a\ge p$, $b\ge q$ and $a+b\ge p+q+2t$, the copy of $B_t(p,q)$ is counted whenever the $t$ page vertices are picked at least once from the neighbors of $u$ and at least once from the neighbors of $v$, the $p$ leaf vertices adjacent to $u$ are picked at least once from the neighbors of $u$ and zero times from the neighbors of $v$, and similarly the $q$ leaf vertices adjacent to $v$ are picked at least once from the neighbors of $v$ and zero times from the neighbors of $u$. Additionally, such a term may appear multiple times, depending on the actual topological index we study. Moreover, we have to consider each term $d(u)^ad(v)^b$ with $a\ge p$, $b\ge q$ and $a+b\ge p+q+2t$.

    The above argument shows that for a given index, determining $w_{t,p,q}$ may be complicated. However, it is clear that $w_{t,p,q}\ge 0$ and that we can assume that $t,p,q\le k$ for some $k$ (in most of the above cases, $k=r$). 

    It is left to show that $w_{t,p,q}$ depends only on $t,p,q$ and the given index from $A_0$, but not on $G$. Recall that when we described how many times we count a copy of $B_t(p,q)$, we considered the terms $d(u)^{a}d(v)^b$, which depend only on the index. We considered such a quantity as taking neighbors of $u$ and $v$ with repetition. The particular sets $U$ and $V$ appear as many times as we can take those sets with repetition. 
    
    For example if $|U|=1$, then $U$ appears only once, by taking that element each time. If $|U|=2$, then $|U|$ appear $2^{a}-2$ times. Indeed, each time we take one of the two elements of $U$, resulting in $2^{a}$, but we have to subtract the cases when the same element appears $a$ times. For larger $U$ a similar but more complicated calculation gives how many times $U$ appear. The actual value is of no interest at this point (it is actually $\sum_{i=0}^{|U|-1} (-1)^{i}\binom{|U|}{i}(|U|-i)^{a}$, as shown in \cite{ger4}). What matters is that it depends only on $a$ and $|U|$, but not on $G$ or $d(u)$. Similarly, the number of times $V$ appears depends only on $a$ and $|V|$. 

    To summarize, the terms depend only of the index in $A_0$, while the number of times a copy of $B_t(p,q)$ is counted at one term depends only on the term and $t,p,q$. This completes the proof.
\end{proof}

We continue with the proof of Theorem \ref{gentu}.

\begin{proof}[Proof of Theorem \ref{gentu}]
    The first statement \textbf{(i)} follows from a result of Morrison, Nir, Norin, Rza{\.z}ewski and Wesolek \cite{mnnrw} that states that for any $H$, $\ex(n,H,K_k)=\cN(H,T(n,k-1))$ if $k$ is large enough. 

To prove \textbf{(ii)}, we use a result of Gy\H ori, Salia, Tompkins and Zamora \cite{gystz}, which states that $\ex(n,P_\ell,P_k)=(1+o(1))\cN(P_\ell,K_{\lfloor (k-1)/2\rfloor,n-\lfloor (k-1)/2\rfloor})$. Since $\ex(n,K_3,P_n)=O(n)$, the number of triangles is hidden in the error term. One can prove \textbf{(iii)} analogously, using a result of Gerbner, Gy\H ori, Methuku and Vizer \cite{ggymv} on $\ex(n,P_\ell,C_{2k})$.

The next statement \textbf{(iv)} follows from that fact that among $K_{2,t}$-free $n$-vertex graph, the F\"uredi graph contains asymptotically the most number of edges \cite{fur}, triangles \cite{AS}, paths of any length \cite{gerpal}. 
In fact, it is enough that the number of copies of $P_4$ is asymptotically maximized there, since the number of copies of other graphs has lower order of magnitude, thus is contained in the error term. 

For indices in $A_0'$, we have to deal in addition to the above graphs with $S_{1,2}$, $S_{2,2}$,$B_1(0,1)$ and $B_1(1,1)$. It was shown in \cite{ger6} that the F\"uredi graph gives the maximum asymptotically for $S_{1,2}$ and $S_{2,2}$, but this does not hold for $B_1(1,1)$. Indeed, the F\"uredi graph contains $O(n^{5/2})$ copies of $B_1(1,1)$, while there are $\Theta(n^3)$ copies of $B_{1,1}$ in the graph we obtain by adding a maximum matching to a star $K_{1,n-1}$. However, the F\"uredi graph contains $(1+o(1))\ex(n,S_{2,2},K_{2,t})=\Theta(n^{7/2})$ copies of $S_{2,2}$. We claim that $\ex(n,H,K_{2,t})=O(n^3)$ for $H$ being $P_2,P_3.P_4,K_3,S_{1,2},B_1(1,0),B_1(1,1)$. For the trees and the triangle, this follows from the preceding satements.
For $B_1(0,1)$, we consider a $K_{2,t}$-free $n$-vertex graph $G$ and pick first the vertices of degree 1 and 3, $O(n^2)$ ways. Then their common neighbor $v$ at most $t-1$ ways. Then the common neighbor of $v$ and the vertex of degree 3, at most $t-1$ ways. This shows that there are $O(n^2)$ copies of $B_1(0,1)$. To get a copy of $B_1(1,1)$, we have to pick another vertex, thus there are $O(n^3)$ copies of $B_1(1,1)$. This shows that the number of copies of each graph other than $S_{2,2}$ is negligible, thus $\ex_Q(n,K_{2,t})=(1+o(1))\ex(n,S_{2,2},K_{2,t})=(1+o(1))\cN(S_{2,2},F(n,t))=(1+o(1))Q(F(n,t))$, completing the proof of \textbf{(v)}.
\end{proof}

Let us continue with the proof of Theorem \ref{gentur}.

\begin{proof}
Observe that to determine $Q(G)$, we only have to deal with counting $B_t(p,q)$ with $2t+p+q\le 3$, i.e., we have to count triangles, triangles with one leaf attached
and double stars $S_{p,q}$ with $p,q\le 2$, that are not $S_{2,2}$. As we have mentioned, $T(n,3)$ contains the most copies of any path \cite{T,gypl,ger} and the triangle \cite{zykov} among $K_4$-free graphs. The same holds for the triangle with a leaf attached, using a result of Gerbner and Palmer \cite{gerpal}. Moreover, the same holds for the graph $H$ we obtain from $P_4$ by adding a vertex adjacent to each vertex of the path, using a result from \cite{gypl}. 

Consider now $B_1(1,1)$. We count the copies in an $n$-vertex $K_4$-free graph $G$ by picking a triangle $uvw$ and two more vertices $x,y$. Observe that $x$ and $y$ are both joined to at most two vertices of the triangle. Let $a=a(uvw)$ be the number of pairs $(x,y)$ that are joined to distinct sets of two vertices in the triangle, $b$ be the number of pairs that are joined to the same set of two vertices, and $c=\binom{n-3}{2}-a-b$ be the number of other pairs. Then the triangle $uvw$ can be extended to a copy of $B_1(1,1)$ at most $3a+2b+2c=2\binom{n-3}{2}+a$ ways. Observe that if $x$ and $y$ are joined to distinct sets of two vertices in the triangle, we obtain a copy of $H$, and each copy of $H$ is counted exactly once this way. Therefore, if we add up $a(uvw)$ for each triangle, we obtain $\cN(H,G)$.
This implies that $\cN(B_1(1,1),G)\le 2\binom{n-3}{2}\cN(K_3,G)+\cN(H,G)$. Each of the terms is maximized by the Tur\'an graph, completing the proof of $\textbf{(i)}$.

Consider now $S_{1,2}$. We use a similar argument. We count the copies in an $n$-vertex $K_4$-free graph $G$ by picking a 4-vertex path first and then picking a fifth vertex. Let $a$ be the number of vertices connected to both middle vertices of the $P_4$, then we have $a+n-4$ ways to obtain an $S_{1,2}$. Each copy of $S_{1,2}$ is counted twice this way. Observe that if we add up $a$ for each $P_4$, then we get $\cN(B_1(1,1),G)$. Therefore, $2\cN(S_{1,2},G)=(n-4)\cN(K_3,G)+\cN(B_1(1,1),G)$. Each of the terms is maximized by the Tur\'an graph, completing the proof of $\textbf{(ii)}$. This also completes the proof of \textbf{(iii)}.
\end{proof}

We continue with the proof of Theorem \ref{trif}, which states that for any index $Q$ in $A_1$, $\ex_Q(n,K_3)=Q(K_{m,n-m})$ for some $m$.
We generalize the proof of the special case $\ex(n,S_{a,b},K_3)$ from \cite{gyww}.

\begin{proof}[Proof of Theorem \ref{trif}]
    Let $G$ be a triangle-free graph on $n$ vertices and let $\Delta$ denote its largest degree. If $\Delta<n/2$, then it is clear that $Q(G)<Q(K_{\lfloor n/2\rfloor,\lceil n/2\rceil})$. Indeed, each edge gives less to the sum and there are less edges in $G$ compared to $K_{\lfloor n/2\rfloor,\lceil n/2\rceil}$.

    Assume now that $\Delta\ge n/2$. We will show that there is a $n/2\le \Delta'\le \Delta$ such that for every edge $uv$ of $G$ and every $i$, we have $f(d(u),d(v))\le f(\Delta',n-\Delta')$.  Clearly $d(u)+d(v)\le n$ because of the triangle-free property, thus $f(d(u),d(v))\le f(d(u),n-d(u))$. Therefore, we can pick $\Delta'$ as an integer $i$ between $n/2$ and $\Delta$ with the largest value of $f(i,n-i)$ to prove the claimed inequality. 

    The number of edges in $G$ is at most $\Delta(n-\Delta)$. Indeed, if $v$ has degree $\Delta$, then the neighbors of $v$ form an independent set, thus they each have degree at most $n-\Delta$. Hence $G$ contains at least $\Delta$ vertices of degree at most $n-\Delta$, while other vertices have degree at most $\Delta$.

    We obtain that the number of edges in $G$ is at most the number of edges in $K_{\Delta',n-\Delta'}$, and each edge gives at most as much as the edges of $K_{\Delta',n-\Delta'}$ to the sum.
\end{proof}

Let us continue with the stability version of the above theorem. We will use the Erd\H os-Simonovits stability theorem \cite{erd1,erd2,simi} for triangles (note that the stronger form of the theorem works for any forbidden non-bipartite graph instead of the triangles). Tt states that for any $\varepsilon>0$ there exists $\delta'>0$ such that if a triangle-free $n$-vertex graph has at least $n^2/4-\delta n^2$ edges, then $G$ can be turned to a complete bipartite graph by adding and removing at most $\epsilon n^2$ edges. We will use the following lemma.

\begin{lemma}[Gerbner \cite{ger5}]\label{deltaersim} 
For any $\varepsilon>0$, there exists $\delta>0$ such that if $G$ is an $n$-vertex $F_2$-free graph with maximum degree $\Delta>n/2$ and at least $\Delta(n-\Delta)-\delta n^2$ edges, then $G$ can be turned to a bipartite graph $G'$ by deleting at most $\varepsilon n^2$ edges. 
\end{lemma}

Now we are ready to prove Lemma \ref{trifstab} that we restate here for convenience.

\begin{lemma*}
     For any index $Q$ in $A_1$, and any $\varepsilon>0$, there exists $\delta>0$ such that if $Q(G)\ge (1-\delta)\ex_Q(n,K_3)$ for an $n$-vertex triangle-free graph $G$, then $G$ can be turned to a complete bipartite graph by adding and removing at most $\epsilon n^2$ edges.
\end{lemma*}

\begin{proof}
    Let $\Delta$ be the maximum degree in $G$. If $\Delta< n/2$, we will use the Erd\H os-Simonovits stability theorem. Thus we are done unless $G$ has less than $n^2/4-\delta' n^2$ edges. We choose $\delta=4\delta'$. In that case $Q(G)<(1-4\delta')\lfloor \frac{n^2}{4}\rfloor f(\lfloor n/2\rfloor,\lceil n/2\rceil)\le (1-\delta)Q(K_{\lfloor n/2\rfloor,\lceil n/2\rceil})$.

    If $\Delta\ge n/2$, then similarly to the proof of Theorem \ref{trif}, we let $\Delta'$ be an integer $i$ between $n/2$ and $\Delta$ with the largest value of $f(i,n-i)$, and observe that $|E(G)|\le \Delta(n-\Delta)\le \Delta'(n-\Delta')$. We use Lemma \ref{deltaersim} to complete the proof if $|E(G)|\ge (1-\delta)\Delta'(n-\Delta')$. If $|E(G)|\le (1-\delta)\Delta'(n-\Delta')$, we obtain a contradiction as in the previous paragraph.
\end{proof}

Let us turn to results in the case $f$ is a polynomial. We need the following simple lemma.

\begin{lemma}\label{ujlem}
    For any index $Q$ in $A_1'$, if $d$ is the degree of the polynomial $f$, then for every edge $uv$, if $G'$ is obtained from $G$ by deleting $uv$, then $Q(G)-Q(G')=O(n^d)$. Moreover, if
    $d(u)=\Theta(n)$ and $d(v)=\Theta(n)$, then $Q(G)-Q(G')=\Theta(n^d)$.
\end{lemma}

\begin{proof}
    Deleting the edge $uv$ we lose a term from the sum defining $Q(G)$, which is $f(d(u),d(v))\le f(n,n)=O(n^d)$. For each other edge containing $u$ or $v$, $f(a,b)$ is replaced by $f(a-1,b)$ for some $a,b\le n$. Clearly for a term of the form $a^pb^q$, we have $a^pb^q-(a-1)^pb^q=O(n^{p+q-1})$, thus for each edge containing one of $u$ and $v$, the $Q$-value decreases by $O(n^{d-1})$. There are less than $2n$ such edges, thus $Q(G)$ decreases by $O(n^d)$.

To see the moreover part, observe that $f(d(u),d(v))=\Theta(n^d)$.
\end{proof}

We continue with the proof of Lemma \ref{ezis} that we restate here for convenience.

\begin{lemma*}
         For any index $Q$ in $A_1'$, any graph $F$ with chromatic number 3 and any $\varepsilon>0$, there exists $\delta>0$ such that if $Q(G)\ge (1-\delta)\ex_Q(n,F)$ for an $n$-vertex $F$-free graph $G$, then $G$ can be turned to a complete bipartite graph by adding and removing at most $\epsilon n^2$ edges.
\end{lemma*}

\begin{proof} We use the triangle removal lemma \cite{rsz}. It states that we can delete $\varepsilon' n^2$ edges from $G$ to obtain a triangle-free graph $G'$ for some $\varepsilon'>0$ of our choice. Having $\varepsilon'\le \varepsilon/2$ and applying Lemma \ref{trifstab} to $G'$ with $\varepsilon/2$, we changed at most $\varepsilon n^2$ edges. We only need to show that we can apply Lemma \ref{trifstab} to $G'$, i.e., $Q(G')\ge (1-\delta')\ex_Q(n,K_3)$ for some $\delta'>0$. Recall that $\ex_Q(n,K_3)=Q(K_{m,n-m})\le\ex_Q(n,F)$ by Theorem \ref{trif} and by the chromatic number of $F$. Therefore, $Q(G)\ge (1-\delta)\ex_Q(n,K_3)$. 

Let $d$ be the degree of the polynomial $f$. Then clearly $\ex_Q(n,K_3)\ge Q(K_{\lfloor n/2\rfloor,\lceil n/2\rceil})=\Omega(n^{d+2})$, while $\ex_Q(n,K_3)\le Q(K_n)=O(n^{d+2})$. We delete $\varepsilon' n^2$ edges, thus using Lemma \ref{ujlem} we have $Q(G')\ge Q(G)-\varepsilon' cn^{d+2}$. By choosing $\varepsilon'$ small enough, we can have that $Q(G')\ge Q(G)-\delta Q(G)$, thus $Q(G')\ge (1-2\delta)Q(G)$. By picking $\delta=\delta'/2$, the proof is complete. 
\end{proof}

Let us continue with the proof of Corollary \ref{corro} that we restate here for convenience.

\begin{corollary*}
    For any index $Q$ in $A_1'$ and any graph $F$ with chromatic number 3 and a color-critical edge, we have that $\ex_Q(n,F)=Q(K_{m,n-m})$ for some $m$, provided $n$ is sufficiently large.
\end{corollary*}

We remark that the proof is quite standard in the case of counting subgraphs.

\begin{proof}
    Let $G$ be an $n$-vertex $F$-free graph with $Q(G)=\ex_Q(n,F)$. By Lemma \ref{ezis}, we can partition $V(G)$ to parts $A$ and $B$ such that there are $o(n^2)$ edges inside $A$ or $B$ and $o(n^2)$ edges are missing between $A$ and $B$. Among such partitions, we pick one with the smallest number of edges inside the parts. In particular, this implies that each vertex has more neighbors in the other part than in its own part. If $|A|=o(n)$ or $|B|=o(n)$, then there are $o(n^2)$ edges in $G$, and as in the proof of Lemma \ref{ezis}, they contribute $o(n^2f(n,n))$ to $Q(G)$, thus $Q(G)<Q(T(n,2))$, a contradiction.

  Assume without loss of generality that deleting the edges inside $A$ would decrease the value of $Q$ by at least as much as deleting the edges inside $B$. 
  Let $uv$ be an edge inside $A$, and let $B'$ be the set of common neighbors of $u$ and $v$ inside $B$. If there is a $K_{|V(F)|,|V(F)|}$ between $A$ and $B'$, then with $u$ and $v$ they form a copy of $F$, a contradiction. If there is no $K_{|V(F)|,|V(F)|}$, then it is well-known \cite{kst} that there are $\Omega(|A||B'|)$ edges missing between $A$ and $B'$. This implies $|B'|=o(n)$. Therefore, $u$ and $v$ are incident to $|B|+o(n)$ edges between parts. Either both $u$ and $v$ are incident to $\Omega(n)$ missing edges between the parts, or one of them, say $u$ is incident to $|B|-o(n)$ edges between parts, thus $v$ is incident to $o(n)$ edges between parts, hence $v$ is incident to $o(n)$ edges altogether.

    Let $A_1$ denote the set of vertices in $A$ incident to $o(n)$ edges of $G$, and let $A_2$ denote the set of vertices in $A\setminus A_1$ incident to $\Omega(n)$ missing edges. Now the number of edges inside $A$ is at most $o(|A_1|n)+\binom{|A_2|}{2}$. The number of edges missing between parts is $\Omega(|A_1|n+|A_2|n)$. 

    Using Lemma \ref{ujlem}, deleting the edges inside $A$ and inside $B$, $Q(G)$ decreases by $o(|A_1|n^{d+1})+O(|A_2|^2n^d)$. Observe that $|A_2|=o(n)$, since otherwise  After that, we can add the missing edges between parts. The first edges incident to a vertex $v\in A_1$ may not increase $Q$ by much, but for each such $v$, after the degree of $v$ reaches $|B|/2$, the remaining linearly many edges each increase $Q$ by $\Theta(n^d)$. This shows that adding the missing edges incident to $A_1$ increases $Q$ by $\Omega(|A_1|n^{d+1})$. Adding the missing edges incident to $A_2$ increases $Q$ by $\Omega(|A_2|2n^{d+1})$. Therefore, $Q$ increases while the resulting graph is bipartite, thus $F$-free, a contradiction completing the proof. \end{proof}

We continue with the proof of Theorem \ref{klik}, which states that the connected $n$-vertex graph with clique number $k$ that contains the least number of copies of $B_t(p,q)$ is the kite graph $Ki(n,k)$, except for $k=2, t=0, p=q=1$.

\begin{proof}[Proof of Theorem \ref{klik}]
    We use induction on $n$. The base case $n=k$ is trivial. The case $n=k+1$ is also trivial once we observe that each connected graph on $k+1$ vertices with clique number $k$ contains $Ki(k+1,k)$ and adding more edges does not decrease the number of copies of any graph.

    Assume now that $n=k+2$. If $t\ge 1$ or $p,q\ge 2$, then $Ki(k+2,k)$ contains the same number of copies of $B_t(p,q)$ as $Ki(k+1,k)$. If $t=0$, $p\le 1$, let $G$ consist of $K_k$ with two leaves $u,v$ attached to distinct vertices and $G'$ consist of $K_k$ with two leaves $u,v$ attached to the same vertex, these are the only other edge-minimal connected graphs with clique number $k$. We compare $\cN(B_t(p,q),Ki(k+2,k))$ to $\cN(B_t(p,q),G)$ and $\cN(B_t(p,q),G')$. Each of these graphs contains $Ki(k+1,k)$, so we will check how many more copies of $B_t(p,q)$ are created. If $p=0$, then $\cN(B_t(p,q),Ki(k+2,k))=\cN(B_t(p,q),Ki(k+1,k))$, unless $q\le 1$. Clearly the case $q=0$ is trivial, and in the case $q=1$, we have one more new copy of $P_3$ in $Ki(k+2,k)$, and at least one more new copy in $G$ and in $G'$. Thus we can assume that $p=1$. If $q=1$, then we have $k-1$ new copies of $P_4$ in $Ki(k+2,k)$, and at least $\binom{k-1}{k-2}$ new copies of $P_4$ in $G$ and $G'$, thus we are done if $k>2$. The case $k=2$ is trivial, since the star does not contain a $P_4$. If $q\ge 2$, then there are $\binom{k-1}{q}$ new copies of $B_t(p,q)$ in $Ki(k+2,k)$. In $G$ and $G'$, by adding $v$, we get at least $(k-1)\binom{k-2}{q}$ new copies by considering those where $v$ is the single leaf adjacent to one of the rootlet vertices, and the other vertices are from the clique. We are done unless $q=k-1$. In this case there is one new copy of $B_t(p,q)$ in $Ki(k+2,k)$. Observe that $B_t(p,q)$ has $k+2$ vertices. Clearly we can find $B_t(p,q)=S_{1,q}$ in $G$, where $u$ and $v$ are leaves adjacent to the same rootlet vertex, and also in $G'$, where $u$ and $v$ are leaves adjacent to different rootlet vertices.
    
Assume now that $n\ge k+3$. Then adding the last vertex to the path in $Ki(n,k)$ creates only $P_2,P_3,P_4$ generalized book graphs, one of each. Clearly adding any more leaf to any connected graph $G_0$ on $n-1$ vertices creates at least one new copy of $P_2$ and $P_3$, and also of $P_4$ unless $G_0$ is a star. In the case $k>2$, $G_0$ cannot be a star, while if $k=2$, the statement is trivial.
\end{proof}

We continue with the proof of Theorem \ref{tranzi}, which states that the connected $n$-vertex graph that contains $H$ as a subgraph and has the least number of copies of $B_t(p,q)$ is the kite graph $Ki(n,H)$, except for $k=2, t=0, p=q=1$.

\begin{proof}[Proof of Theorem \ref{tranzi}]
    The proof goes similarly to the proof of Theorem \ref{klik}. We assume familiarity with that proof. Let $d$ be the degree of vertices of $H$. We use induction, the base cases $n\le |V(H)|+1$ are trivial. In the case $n=|V(H)|+2$, the subcases $t\ge 1$ or $p,q\ge 2$ or $t=0=p$ work the same way. Note however that $G'$ is not unique, as there may be no automorphism that brings two fixed distinct vertices to two other fixed distinct vertices. For example, if $H=C_k$, then the two leaves may be connected to the same vertex of $H$ or to distinct vertices, but then we still have to tell how far those vertices are to get a complete description of the graph. This does not matter in the above simple cases.
    
    If $t=0$ and $p=q=1$, then we have $d$ new copies of $P_4$ in $Ki(k+1,H)$, and at least $d(d-1)$ new copies of $P_4$ in the other graphs, as we add a new leaf $v$ connected to a vertex $v_0$ of $H$, then we pick a neighbor of $V_0$ in $H$ $d$ ways and a neighbor of that vertex. We are done unless $d=1$, but the only 1-regular connected graph is $K_2$.

    If $t=0$, $p=1$ and $q\ge 2$, then there are $\binom{d}{q}$ new copies of $B_t(p,q)$ in $Ki(k+2,H)$. If we add a new leaf $v$ connected to a vertex of $H$, we get at least $d\binom{d-1}{q}$ new copies by considering those where $v$ is the single leaf adjacent to one of the rootlet vertices, and the other vertices are from the copy of $H$. We are done unless $q=d$. In this case there is one new copy of $B_t(p,q)$ in $Ki(k+2,H)$. Clearly we can find a copy of $S_{1,q}$ using the new leaf $v$ and $d-1$ neighbors of $v_0$ in $H$ as the leaves connected to one of the centers and the remaining neighbor $v_1$ of $v_0$ as the other center. If $H$ is not a clique, $v_1$ has another neighbor in $H$ that we can use as the remaining leaf.

    The case $n\ge |V(H)|+3$ works the same way as in the proof of Theorem \ref{klik}. Here we use that the only vertex-transitive subgraph of a star is $K_2$.
\end{proof}

\vskip 0.3truecm

\textbf{Funding}: Research supported by the National Research, Development and Innovation Office - NKFIH under the grants  FK 132060 and KKP-133819.

\vskip 0.3truecm


\end{document}